\documentclass[11pt,twoside, reqno]{amsart}

\usepackage{amsmath}
\usepackage{amsthm}
\usepackage{amsfonts, amssymb}
\usepackage{mathrsfs}
\usepackage[all]{xy}
\usepackage{url}

\usepackage{latexsym}
\usepackage[dvips]{graphicx}

\theoremstyle{plain}
\newtheorem{lema}{Lemma}
\newtheorem{prop}[lema]{Proposition}
\newtheorem{teo}[lema]{Theorem}

\newtheorem*{intro1}{Theorem \ref{teo2conexo}}
\newtheorem*{intro2}{Theorem \ref{main}}
\newtheorem*{intro3}{Corollary \ref{proba}}
\newtheorem{coro}[lema]{Corollary}
\theoremstyle{remark}

\newtheorem{obs}[lema]{Remark}
\theoremstyle{definition}
\newtheorem{defi}[lema]{Definition}
\newtheorem{ej}[lema]{Example}

\newcommand{\Z}{\mathbb{Z}}
\newcommand{\N}{\mathbb{N}}

\newcommand{\ost}{\overset{\circ}{\textrm{st}}}
\newcommand{\st}{\textrm{st}}

\newcommand{\lk}{\textrm{lk}}

\begin{document}

\title[Connectivity of ample, conic and random simplicial complexes]{Connectivity of ample, conic and random simplicial complexes}

\author[J.A. Barmak]{Jonathan Ariel Barmak $^{\dagger}$}

\thanks{$^{\dagger}$ Researcher of CONICET. Partially supported by grant PICT-2017-2806, PIP 11220170100357CO, UBACyT 20020160100081BA}

\address{Universidad de Buenos Aires. Facultad de Ciencias Exactas y Naturales. Departamento de Matem\'atica. Buenos Aires, Argentina.}

\address{CONICET-Universidad de Buenos Aires. Instituto de Investigaciones Matem\'aticas Luis A. Santal\'o (IMAS). Buenos Aires, Argentina. }

\email{jbarmak@dm.uba.ar}

\begin{abstract}
A simplicial complex is $r$-conic if every subcomplex of at most $r$ vertices is contained in the star of a vertex. A $4$-conic complex is simply connected. We prove that an $8$-conic complex is $2$-connected. In general a $(2n+1)$-conic complex need not be $n$-connected but a $6^n$-conic complex is $n$-connected. This extends results by Even-Zohar, Farber and Mead on ample complexes and answers two questions raised in their paper. Our results together with theirs imply that the probability of a complex being $n$-connected tends to $1$ as the number of vertices tends to $\infty$. Our model here is the medial regime. 
\end{abstract}



\maketitle


\section{Introduction} \label{content}

Let $r\in \N$. A non-empty simplicial complex $K$ with vertex set $V_K$ is said to be $r$-ample if for each subset $U\subseteq V_K$ of at most $r$ vertices, and for each subcomplex $L\leqslant K$ with $V_L\subseteq U$, there exists a vertex $v\in K$ such that the simplices in $\lk_K(v)$ contained in $U$ are exactly those of $L$. In other words the link $\lk_{K(U\cup \{v\})}(v)$ of $v$ in the full subcomplex of $K$ induced by $U\cup \{v\}$ is $L$. A complex which is $r$-ample for every $r\ge 1$ is said to be $\infty$-ample. Of course, an $\infty$-ample complex must be infinite and it is proved in \cite{FMS} that up to an isomorphism there is a unique one such complex $K$ with countable many vertices. This complex has the following resilience property: any subcomplex of $K$ obtained by removing finitely many simplices is isomorphic to $K$.

For every $r\ge 1$ there are examples of finite complexes which are $r$-ample. Probabilistic arguments have been used to prove their existence and also deterministic constructions are available. The motivation behind this notion are potential applications in network science. A biological, social or computational system can be studied through the interaction among their components. For many years these systems have been modeled by graphs, taking into account only the interactions between pairs. More recently ecological, neurological and social systems have been modeled by different structures, in order to consider the interactions which occur among several units. For instance, when studying the spread of a disease in a community the transmission is usually assumed to occur through pairwise interaction. However for some diseases this is not as simple, and the exposure of a healthy individual to several infectious people has to be considered. Also, in social contagion phenomena, like diffusion of rumors or adoption of norms, the transmission needs one person to be in contact with multiple sources (see \cite{Iac}). A structure which models multiple interactions is given by simplicial complexes. In particular, complex dynamical systems can be modeled in this way and Topology has become part of the new multidisciplinary field of Network Science. For a comprehensive report of the state-of-the-art of complex networks beyond pairwise interaction see \cite{BCI}.

Being a finite approximation to $\infty$-ampleness one expects $r$-ample complexes to enjoy some resilience property as well, and it is proved in \cite{EZFM} that removing a ``small'' part from an $r$-ample complex leaves an $(r-k)$-ample complex for some $k\ge 0$ (on which the definition of ``small'' depends). For this reason ample complexes can be useful to model stable networks.

Ampleness is related to connectivity. For instance, $2$-ample complexes are connected. In fact if $v,w$ are vertices of a $2$-ample complex, there has to be a third vertex adjacent to both. It is not hard to prove that $4$-ampleness implies simple connectivity. This is proved in \cite{EZFM} and an example is given which shows that $2$-ampleness does not guarantee simple connectivity. The authors of \cite{EZFM} say that they do not know examples of $3$-ample complexes which are not simply connected. In \cite[Theorem 4.2]{EZFM} it is proved that an $18$-ample complex is $2$-connected. In the end of Section 4 in \cite{EZFM} the authors implicitly state a conjecture: ``We tend to believe that in general, for every $k\ge 1$ there exists $r(k)$ such that every $r$-ample simplicial complex is $k$-connected provided that $r\ge r(k)$. We know that $r(1)\le 4$ and $r(2) \le 18$.''  

The first result of this article is that for each $n\ge 0$ there is an infinite $(2n+1)$-ample simplicial complex which is not $n$-connected, and in particular $3$-ampleness does not imply simple connectivity. The most important results of our paper are a proof of a strong version of the conjecture stated above and an improvement of the bound $r(2)\le 18$, together with their consequences to the connectivity of random simplicial complexes.

We introduce the notion of an $r$-conic simplicial complex which is weaker than $r$-ampleness. A simplicial complex is $r$-conic if every subcomplex of at most $r$ vertices is contained in the closed star of some vertex. We prove the following results

\begin{intro1} 
Every $8$-conic simplicial complex is $2$-connected.
\end{intro1}

\begin{intro2}
Let $n\in \N$. If a simplicial complex is $6^{n}$-conic, then it is $n$-connected.
\end{intro2}

One remarkable corollary that follows from Theorem \ref{main} and a result of Even-Zohar, Farber and Mead in \cite{EZFM} concerns the connectivity of random simplicial complexes. In the last 15 years Random Topology has grown driven by real life and pure mathematics applications. On one hand randomness models nature. A variety of random network models have been used to describe biological, technological and social systems. Most of them use graphs (not higher dimensional structures) and various complexities of randomness \cite{Sci, Nat}. On the other hand the probabilistic method can be used to prove existence results (see \cite{Kahleviejo}). A general question involving random complexes has this form: how does a particular homotopy/combinatorial invariant behaves in a random complex on $n$ vertices when $n\to \infty$. Different invariants like homology groups, homotopy groups, collapsibility, embeddability, asphericity, have been studied for distinct models of randomness: the Linial-Meshulam-Wallach model, the clique complex model by Kahle, geometric models based on the Vietoris-Rips or the \v Cech complex. Arguably a more natural model, which generalizes the first two models above is the multiparameter model (mentioned by Kahle in \cite{Kahleviejo} and first studied by Costa and Farber in \cite{CF}), in which simplices are added successively in each dimension with some probability. A complete survey on random complexes can be found in \cite{Kahlenuevo}. If the probability $p_{\sigma}$ of each simplex being added in the multiparameter model lies in an interval $(\epsilon, 1-\epsilon)$ for every $\sigma$, we are in the presence of what Farber and Mead called the \textit{medial regime} in \cite{FM}. It is proved there that in the medial regime the probability of a complex being simply connected tends to $1$ as $n\to \infty$. Also, the Betti numbers of a random simplicial complex vanish in dimension smaller than or equal to a fixed dimension $d$ with probability $1$ as $n\to \infty$. Nothing is proved there about torsion in homology in degrees below $d$. Finally in \cite{EZFM} it is proved that a random complex is $2$-connected with probability $1$ as $n\to \infty$.

As a direct application of Theorem \ref{main} and a result by Even-Zohar, Farber and Mead in \cite{EZFM} we will deduce the following

\begin{intro3}
Let $d\ge 0$. In the medial regime the probability of a complex in $n$ vertices being $d$-connected tends to $1$ as $n\to \infty$. 
\end{intro3}

\section{Simple connectivity and conic complexes}
 
Throughout the paper we will assume the reader is familiar with basic notions and results of Algebraic Topology, such as homotopy groups, homology, the Mayer-Vietoris sequence, the Hurewicz theorem and the fact that a map $S^k\to X$ extends to $D^{k+1}$ if and only if it is null-homotopic; and of polyhedra, such as the edge-path group of a simplicial complex, links, stars, combinatorial manifolds, pseudomanifolds, homology manifolds, simplicial approximations, barycentric subdivisions of regular CW-complexes, shellability. Standard references for these are \cite{Gla, Mun, Spa}, but we will include more specific references when needed. Sometimes we will distinguish between simplicial complexes or maps and their geometric realizations but other times we will identify them.

\begin{teo} \label{contraejemplo}
For every $n\ge 0$ there exists an infinite simplicial complex $K$ which is $(2n+1)$-ample and which is not $n$-connected. In particular, $3$-ampleness does not imply simple connectivity.  
\end{teo}
\begin{proof}
Let $K_0=(S^0) ^{* (n+1)}$ be the join of $n+1$ copies of the discrete complex on $2$ vertices. This is an $n$-dimensional sphere. For each $i\ge 0$, the complex $K_{i+1}$ is constructed from $K_i$ by attaching a cone with basis $L$ for every subcomplex $L$ of $K_i$ with at most $2n+1$ vertices. It is clear then that $K=\bigcup\limits_{i\ge 0} K_i$ is $(2n+1)$-ample. We claim that the fundamental class $[K_0]$ is nontrivial in $H_n(K)$, and in particular $K$ is not $n$-connected. We prove something stronger: $[K_0]$ is nontrivial in the clique closure $c(K)$ of $K$, which is obtained from $K$ by adding as simplices all the finite subsets of vertices which are pairwise adjacent in $K$. Suppose that $[K_0]$ is trivial in $H_n(c(K))$. Since $K_0$ is clique ($c(K_0)=K_0$) then there exists $i\ge 0$ such that $[K_0]$ is nontrivial in $H_n(c(K_i))$ but trivial in $H_n(c(K_{i+1}))$. Moreover, since $c(K_{i+1})$ is obtained from $c(K_i)$ by adding finitely many vertices, there is a clique complex $c(K_i)\leqslant C< c(K_{i+1})$ and a clique subcomplex $L\leqslant C$ with at most $2n+1$ vertices such that $[K_0]$ is nontrivial in $H_n(C)$ but trivial in the $n$th homology group of $C\cup vL$, the complex obtained from $C$ by attaching a cone with basis $L$. By the Mayer-Vietoris sequence $$H_n(L)\to H_n(C) \to H_n(C\cup vL),$$ $[K_0]\in H_n(C)$ is in the image of the map $H_n(L)\to H_n(C)$ induced by the inclusion. But a clique complex with less than $2n+2$ vertices has trivial homology in degree $n$ (see for instance \cite[Theorem 1.1]{Mes}), a contradiction.
\end{proof}

For $n=1$ the proof shows that there is a $3$-ample complex containing a $4$-cycle which is nontrivial in $H_1$. 




\begin{defi}
Let $r\in \Z_{\ge -1}$. We say that a simplicial complex $K$ is \textit{$r$-conic} if every subcomplex $L\leqslant K$ of at most $r$ vertices is contained in a simplicial cone, or, equivalently, in the (closed) star $\st_K (v)$ of a vertex $v\in K$.
\end{defi}

Of course, $r$-amplness implies $r$-conicity. Note that by definition every complex is $(-1)$-conic and a complex is $0$-conic if and only if it is non-empty, and $1$-conicity is equivalent to $0$-conicity.

\begin{ej}
A simplex is $r$-conic for every $r$. Moreover, a finite complex is $r$-conic for every $r$ if and only if it is a cone.
\end{ej}

\begin{ej}
The Cs\'asz\'ar polyhedron is a triangulation of the torus which is $2$-conic.
\end{ej}

It is proved in \cite{EZFM} that an $r$-ample simplicial complex must have at least $2^{O(\frac{2^r}{\sqrt{r}})}$ vertices and it is not easy to construct explicit examples. On the other hand there are examples of $r$-conic complexes for vertex sets of arbitrary cardinality and we can build many using the following remark.

\begin{obs} \label{join}
The join $K_1*K_2$ of an $r_1$-conic complex $K_1$ and an $r_2$-conic complex $K_2$ is $(r_1+r_2+1)$-conic. Indeed, this is trivial if either $K_1$ or $K_2$ is empty. Otherwise, if $L$ is a subcomplex of $K_1*K_2$ with at most $r_1+r_2+1$ vertices, then it is a subcomplex of a join $L_1*L_2$ with the same vertex set and $L_i\leqslant K_i$ for $i=1,2$. Without loss of generality assume $L_1$ has at most $r_1$ vertices. Then $L_1\leqslant \st_{K_1}(v)$ for some $v\in K_1$ and then $L\leqslant L_1*L_2\leqslant \st_{K_1*K_2}(v)$. 
\end{obs}

\begin{ej}
Let $K=(S^0)^{*(n+1)}$. By Remark \ref{join}, $K$ is a finite complex which is $(2n+1)$-conic and $|K|$ is homeomorphic to $S^n$ (in particular $(n-1)$-connected and not $n$-connected). This can be compared with the construction in Theorem \ref{contraejemplo}. The complex $K$ is now finite.
\end{ej}


A $4$-ample simplicial complex is simply connected, but this holds for $4$-conic complexes as well. This can be deduced from \cite{EZFM} already, but we give a different proof here. Even though the proof is easy, we include it as a first example of one of the main differences with \cite{EZFM}: the combinatorial definition of $n$-connectivity used there is sometimes hard to handle, so we will use the original definition, a space $X$ is $n$-connected if for every $0\le k\le n$, $\pi_k(X)=0$, or equivalently $X$ is non-empty and each continuous map $S^k\to X$ is null-homotopic. For the case $n=1$ we will use the description of the fundamental group of a simplicial complex given by the edge-path group (see \cite{Spa} for definitions). Alternatively one could take an arbitrary map $S^1\to |K|$, a simplicial approximation $\varphi:L\to K$ for a cycle graph $L$, and copy the proof below by induction in the number of vertices of $L$ to show that $\varphi$ is null-homotopic.

\begin{prop}
A $4$-conic simplicial complex is simply connected.
\end{prop}
\begin{proof}
We will prove that the edge-path group is trivial. Given a closed edge-path $\xi=(v_0,v_1)(v_1,v_2)\ldots (v_{n-1},v_0)$, if $n\le 4$, it is contained in the star of a vertex, so it is trivial in the edge-path group. If $n=5$, we use only $3$-conicity to show that the subcomplex generated by $v_0$ and $v_2v_3$ is in the star of a vertex $v$. By induction $\xi_1=(v_0,v_1)(v_1,v_2)(v_2,v)(v,v_0)$ and $\xi_2=(v_0,v)(v,v_3)(v_3,v_4)(v_4,v_0)$ are trivial, so $\xi\sim \xi_1\xi_2$ is trivial. For $n\ge 6$, the $2$-conicity implies that there is a vertex adjacent to $v_0$ and $v_3$, so $\xi$ is equivalent to a concatenation of two shorter closed edge-paths.
\end{proof}

\section{Starrings and $2$-connectivity}

\begin{defi}
Let $r\in \N$, let $M$ be a simplicial complex and let $D$ be a subcomplex of $M$ with at most $r$ vertices which is a triangulation of a disk $D^d$. In particular $D$ is a pseudomanifold (see \cite{Spa} for background on pseudomanifolds). Let $\partial D$ be its boundary. Suppose that every coface of a simplex of $D$ not in $\partial D$ is already in $D$. Then there is a simplicial complex $\widetilde{M}$ obtained from $M$ by removing all the simplices of $D$ which are not in $\partial D$ and attaching a cone $w\partial D$ for some vertex $w$ not in $M$. We say that $\widetilde{M}$ has been obtained from $M$ by an \textit{$r$-starring} and that $D$ can be \textit{$r$-starred} in $M$. If by a sequence of $r$-starrings we can transform $M$ into a complex $L$, we say that $M$ \textit{$r$-stars} to $L$.
\end{defi}

The notion of a starring at a combinatorial (=PL) ball $D$ in the sense of Alexander \cite[Section V]{Ale} is an $r$-starring in our sense provided $D$ has at most $r$ vertices.

\begin{obs}
If $M$ is a combinatorial (=PL) manifold of dimension $n$ and $D\leqslant M$ is a triangulation of the $n$-disk $D^n$, then every coface of a simplex in $D$ and not in $\partial D$ is in $D$. That is, we do no need to check this condition to prove that $D$ can be starred. This will always be the case in the results below.

Indeed, since $M$ is a combinatorial manifold, a $k$-simplex $\sigma \in  D \smallsetminus \partial D$ satisfies that $\lk_M(\sigma)$ is an $(n-k-1)$-(combinatorial) sphere. On the other hand since $|D|$ is a (topological) $n$-manifold with boundary $|\partial D|$, $(|D|,|\partial D|)$ is a relative homology $n$-manifold and $\lk_D(\sigma)$ has the homology of an $(n-k-1)$-sphere (see \cite[Theorem 63.2]{Mun}). Any proper subcomplex of the sphere $\lk_M (\sigma)$ has trivial homology in degree $n-k-1$, thus $\lk_D(\sigma)=\lk_M(\sigma)$ as we wanted to prove.

The remark remains true if $M$ is just a triangulation of an $n$-manifold, but we will not need that here. In that case $\lk_M (\sigma)$ has the homology of $S^{n-k-1}$ and it is also a homology $(n-k-1)$-manifold, so it is a pseudomanifold (see \cite{Mun}, p. 377) and then also here a proper subcomplex of $\lk_M(\sigma)$ has trivial homology in degree $n-k-1$.
\end{obs}

As usual, if two maps $f,g$ are homotopic we write $f\simeq g$.

\begin{lema} \label{disco}
Let $r\in \N$ and let $K$ be an $r$-conic complex. Let $M$ be a simplicial complex and $\varphi: M\to K$ a simplicial map. If $\widetilde{M}$ is obtained from $M$ by an $r$-starring, then there is a homeomorphism $h:|\widetilde{M}| \to |M|$ and a simplicial map $\widetilde{\varphi}:\widetilde{M}\to K$ such that $|\varphi|h \simeq |\widetilde{\varphi}|$. In particular $|\varphi|$ is null-homotopic if and only if $|\widetilde{\varphi}|$ is null-homotopic.  
\end{lema}
\begin{proof}
By hypothesis there exists a subcomplex $D$ of $M$ with at most $r$ vertices which is a triangulation of a disk, say of dimension $d$, and $\widetilde{M}$ is obtained from $M$ by removing all the simplices of $D$ which are not in $\partial D$, and attaching a cone $w\partial D$. Since $D$ triangulates $D^d$, $\partial D$ triangulates $S^{d-1}$. Thus the cone $w\partial D$ triangulates $D^d$ and there is a homeomorphism $|w\partial D| \to |D|$ which is the identity on $|\partial D|$ and extends to a homeomorphism $h: |\widetilde{M}|\to |M|$.



Now, since $\varphi (D)$ has at most $r$ vertices, it is contained in $\st (v)$ for some $v\in K$. Define $\widetilde{\varphi}:\widetilde{M}\to K$ by $\widetilde{\varphi}(w)=v$ and $\widetilde{\varphi}(u)=\varphi (u)$ for every $u\neq w$. Clearly $\widetilde{\varphi}$ is simplicial. In order to see that $|\varphi|h \simeq |\widetilde{\varphi}|$, it suffices to show that they are homotopic relative to $|\partial D|$ when restricted to $|w\partial D|$. But this is clear since both $|\varphi|h$ and $|\widetilde{\varphi}|$ map $|w\partial D|$ to $|\st (v)|$, which is contractible. Concretely, the map $H: |\partial D|\times I \cup |w \partial D| \times \{0,1\} \to |\st (v)|$ defined as $|\varphi|h$ in the bottom $|w\partial D|\times \{0\}$, as $|\widetilde{\varphi}|$ in the top $|w\partial D|\times \{1\}$ and as any (both) of those in the cylinder $|\partial D|\times I$, extends to all $|w \partial D|\times I$ since the domain of $H$ is homeomorphic to $S^d$ and $|\st (v)|$ is contractible.
\end{proof}

\begin{teo} \label{teo2conexo}
Every $8$-conic simplicial complex is $2$-connected.
\end{teo}
\begin{proof}
Let $K$ be an $8$-conic complex. We already know that $K$ is simply connected. Let $f:S^2\to |K|$ be a continuous map. We want to prove that $f$ is null-homotopic. Let $\mathcal{U}=\{f^{-1}(\ost (v)) | \ v \in K\}$. As usual $\ost (v)$ stands for the open star of $v$. Recall that a triangulation $(M,h)$ of $S^2$ (that is a simplicial complex $M$ and a homeomorphism $h:|M|\to S^2$) is said to be finer than $\mathcal{U}$ if for every vertex $w\in M$ there exists $U\in \mathcal{U}$ such that $h(\ost (w))\subseteq U$. A simplicial approximation $\varphi :M \to K$ of $fh:|M|\to |K|$ exists if and only if $(M,h)$ is finer than $\mathcal{U}$ (\cite[Theorem 3.5.6]{Spa}). Instead of dealing with a particular triangulation of $S^2$ and its barycentric subdivisions as custom, which would lead to a weaker result, we will work with a different family of triangulations of $S^2$.

Given $n\ge 3, m\ge 0$ we define the simplicial complex $M_{n,m}$ as follows. We divide a square $I\times I$ in $n\times m$ squares and subdivide each of these small squares in two triangles. We identify the two vertical sides of $I\times I$ to obtain a triangulation of $S^1\times I$. Finally we attach two cones to the cylinder, one with base $S^1\times \{0\}$ and the other with base $S^1\times \{1\}$. More precisely, the vertices of $M_{n,m}$ are $0,1$ and those of the form $(i,j)$ for $0\le i \le n-1, 0\le j\le m$. The complex is homogeneous $2$-dimensional and its $2$-simplices are the following: $\{(0,j), (1,j), (0, j+1)\}$, $\{(0,j+1), (1, j), (1,j+1)\}$ for $0\le j\le m-1$; $\{(i,j), (i+1,j), (i+1,j+1)\}$, $\{(i,j), (i+1,j+1), (i,j+1)\}$ for $1\le i \le n-1$ (identifying $i+1=n$ with $i+1=0$) and $0\le j\le m-1$; $\{0, (i,0), (i+1,0)\}$, $\{1, (i,m), (i+1,m)\}$ for $0\le i\le n-1$. See Fig. \ref{mn} (a).

\begin{figure}[h] 
\begin{center}
(a) \includegraphics[scale=0.45]{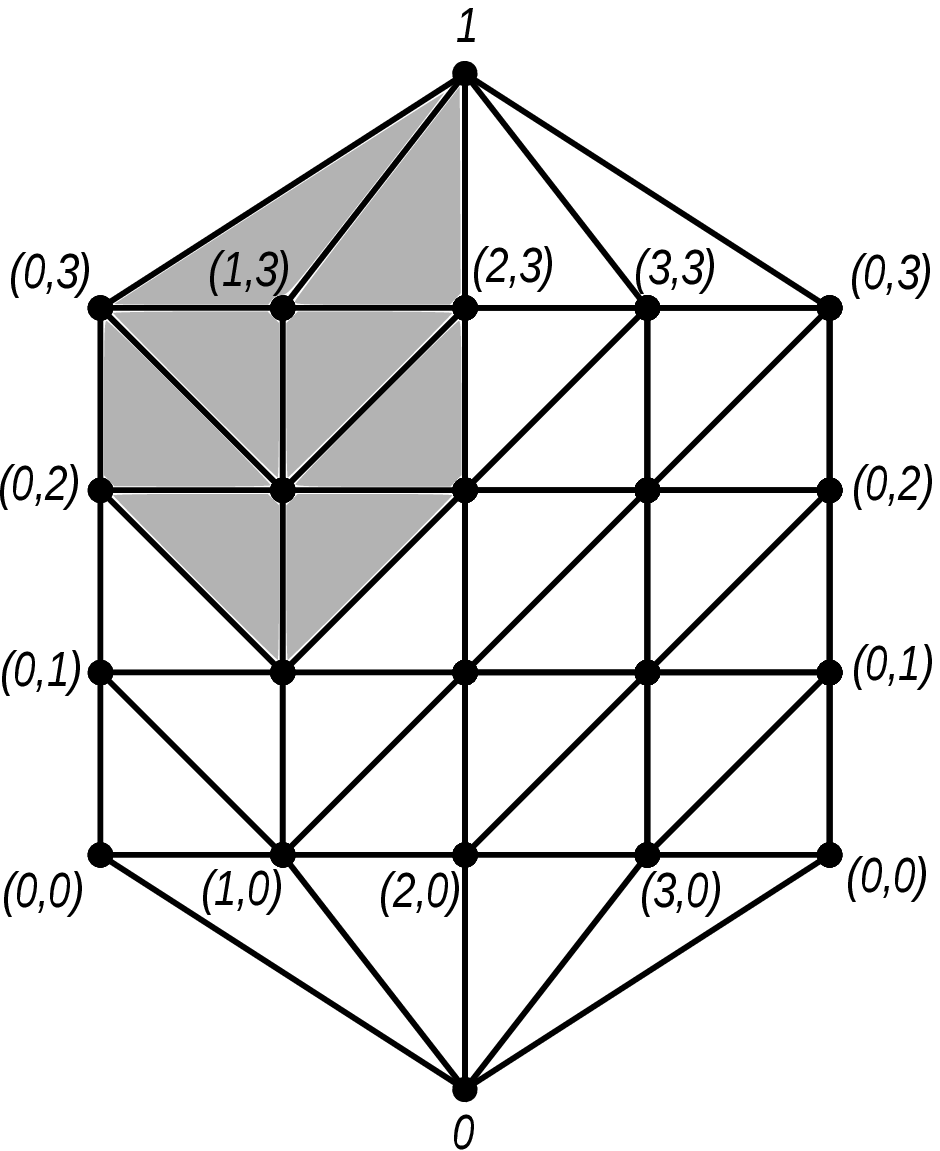} 
(b) \includegraphics[scale=0.45]{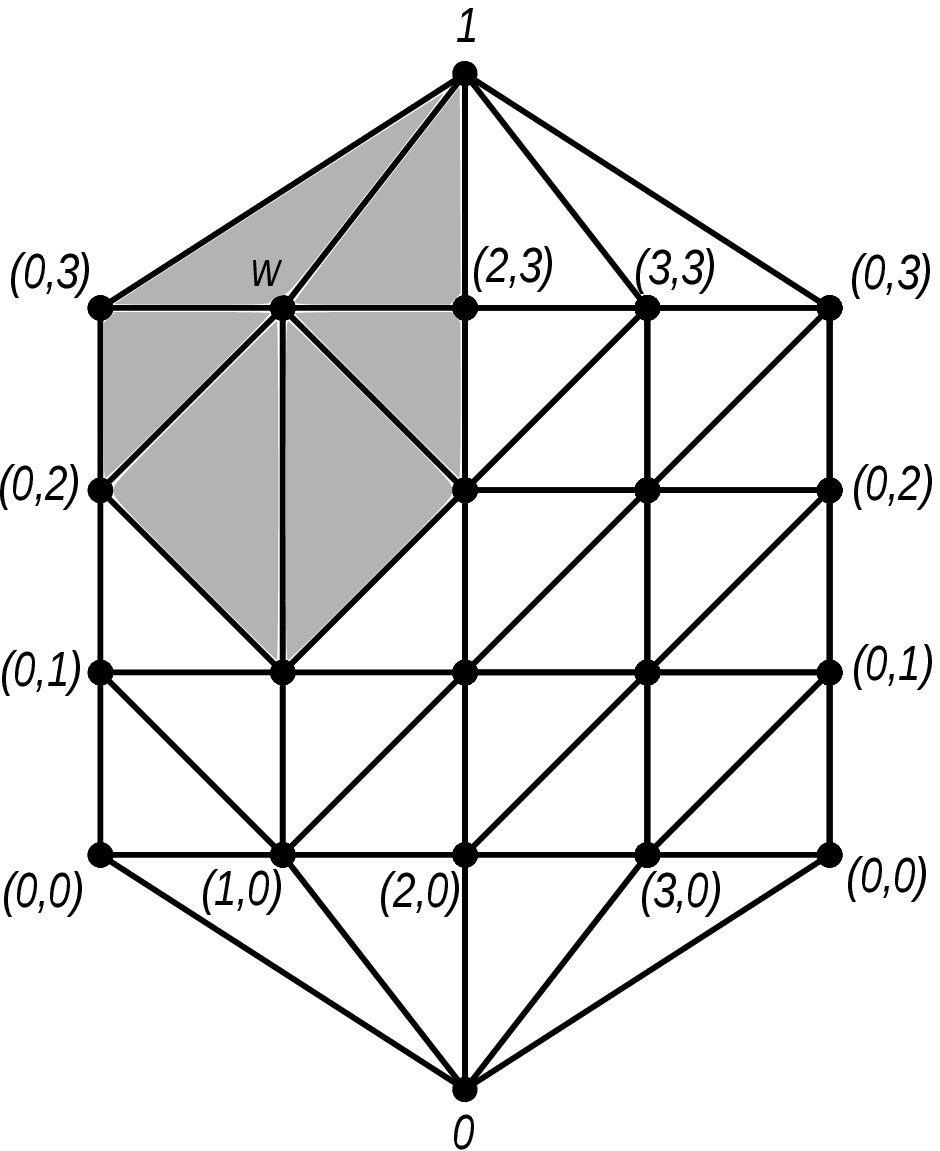} 
(c) \includegraphics[scale=0.45]{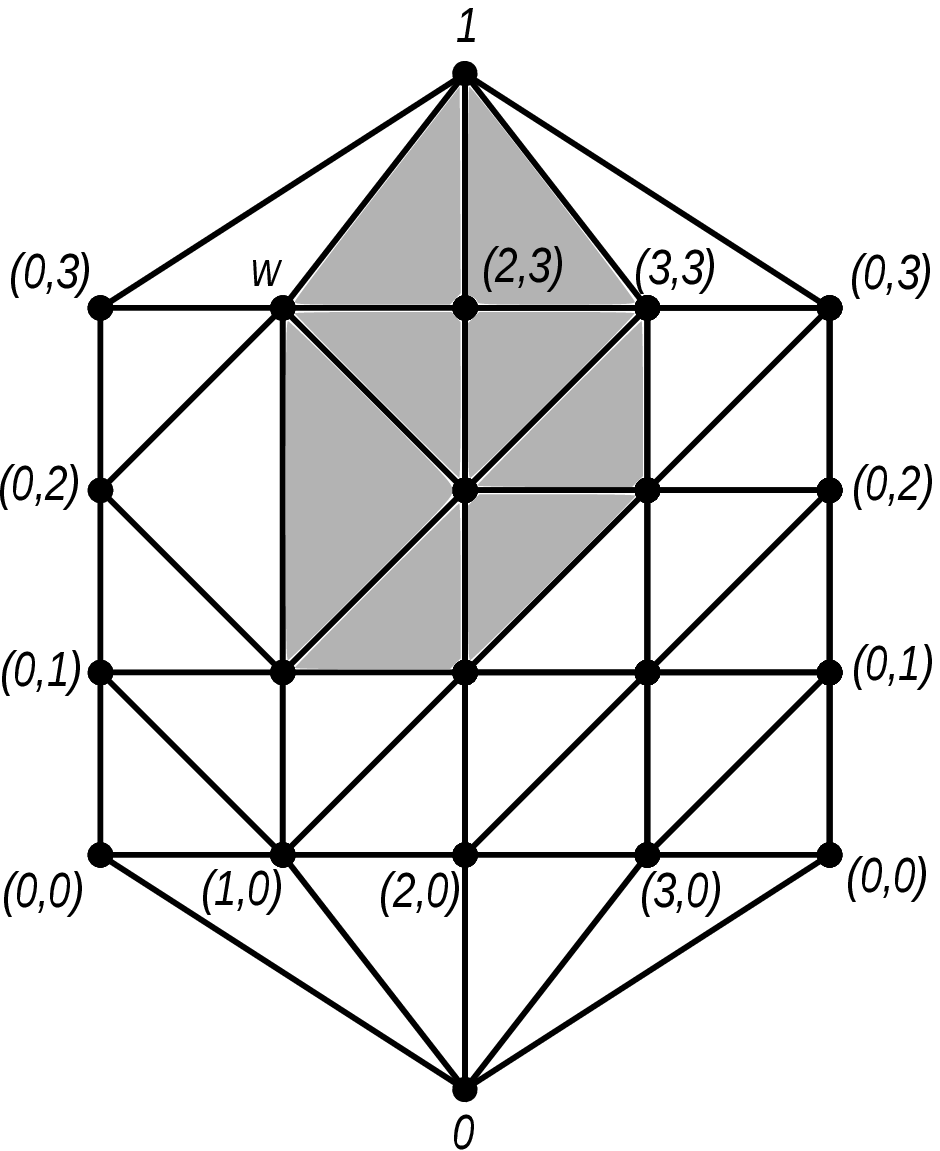}

(d) \includegraphics[scale=0.45]{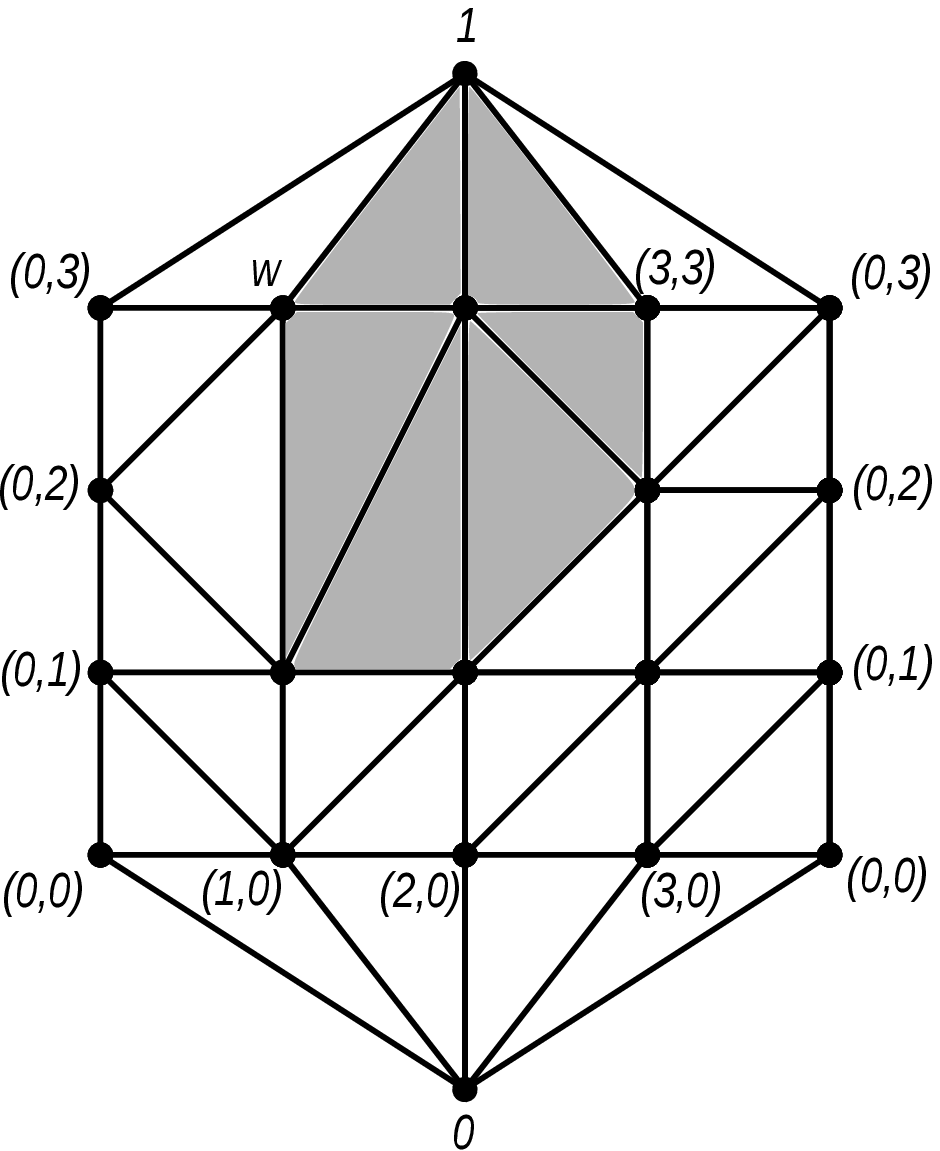}
(e) \includegraphics[scale=0.45]{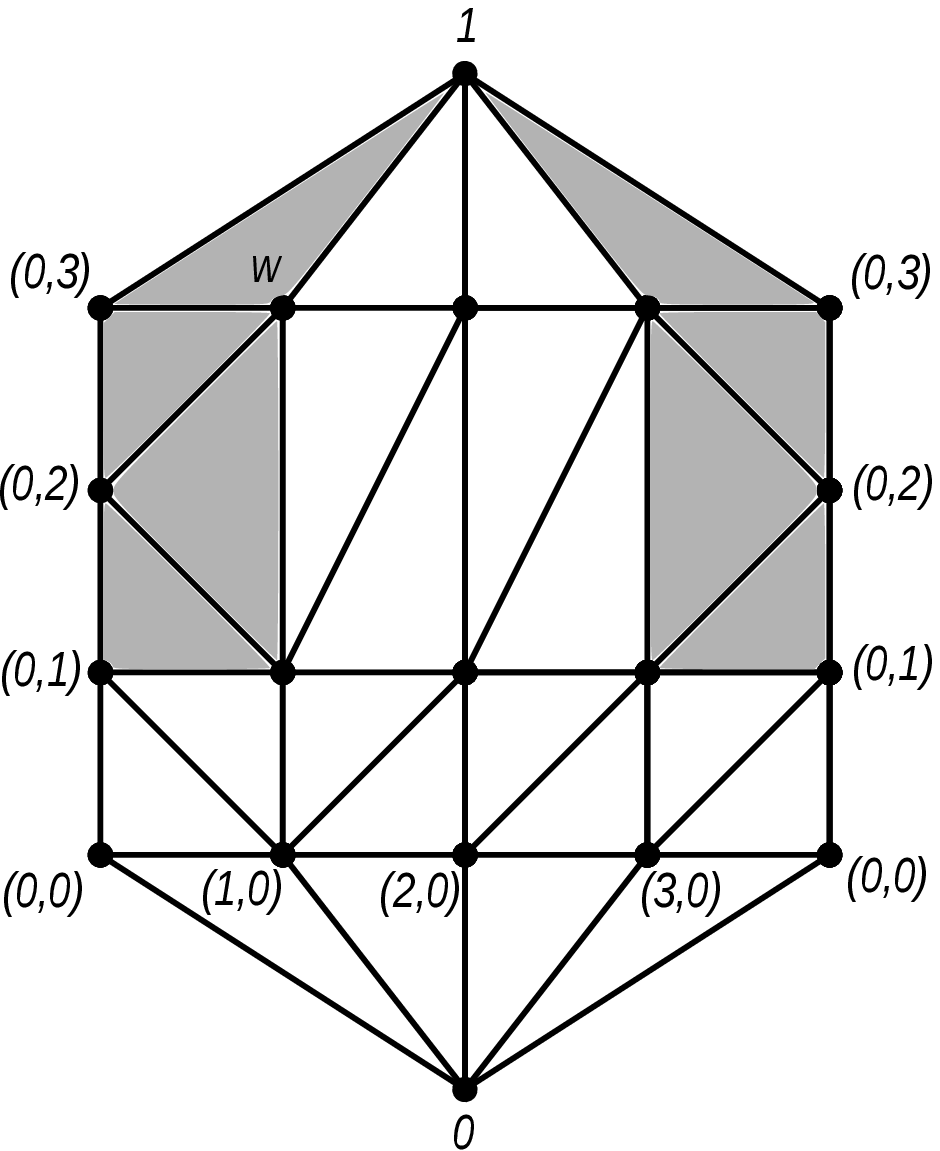}
(f) \includegraphics[scale=0.45]{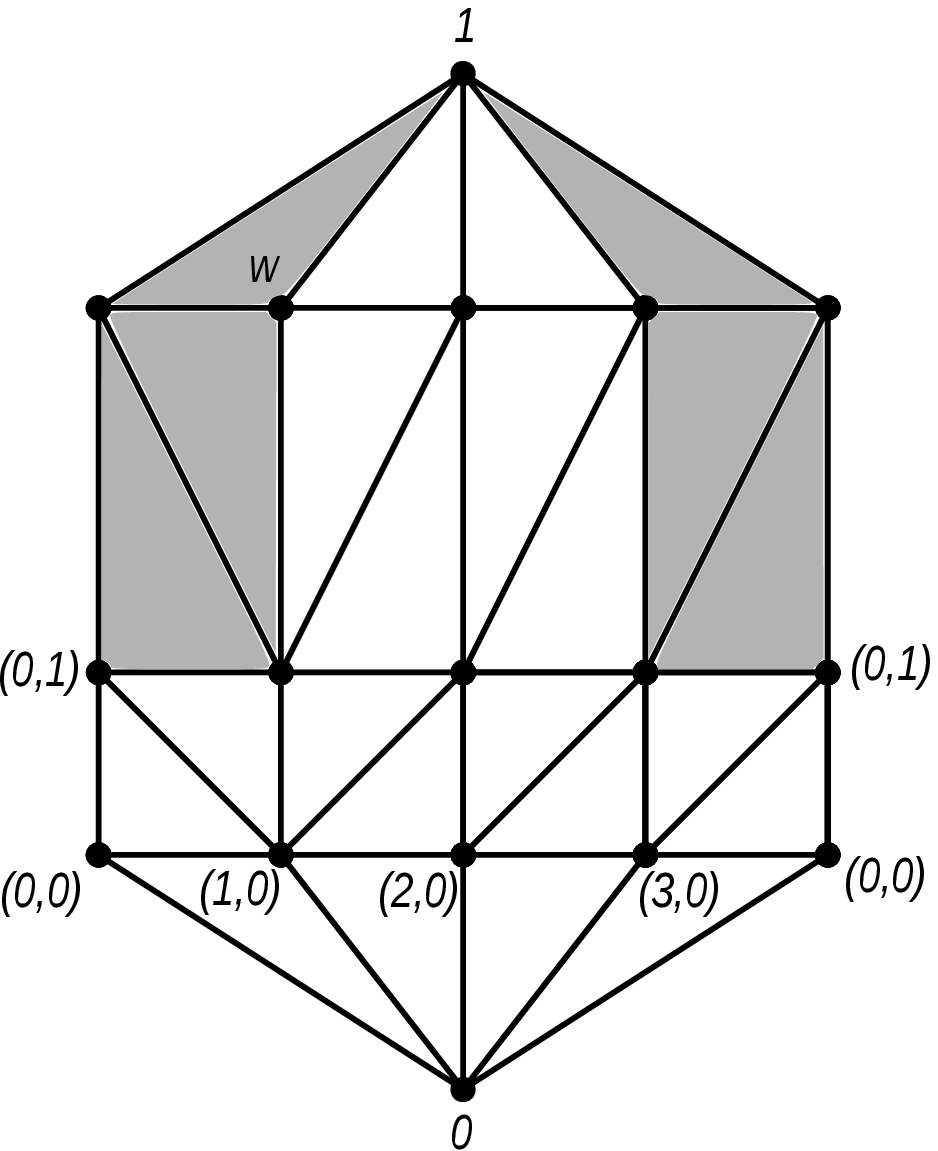}
\caption{The complex $M_{n,m}$ for $n=4,m=3$ and a sequence of $8$-starrings ending in $M_{n,m-1}$.}\label{mn}
\end{center}
\end{figure}

We claim that there exists $n\ge 3$, $m\ge 0$ and a homeomorphism $h: M_{n,m}\to S^2$ such that $(M_{n,m},h)$ is finer than $\mathcal{U}$. Indeed, $\mathcal{U}$ has a Lebesgue number $\delta$ and we can find parallels and meridians in $S^2$ in such a way that the diameter of each spherical rectangle (formed by two consecutive parallels and two consecutive meridians) and each spherical triangle (formed by two consecutive meridians and the top or the bottom parallel) is smaller than $\delta /2$ (see Fig. \ref{esf}).

\begin{figure}[h] 
\begin{center}
\includegraphics[scale=0.6]{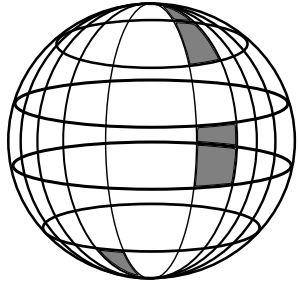}
\caption{Small rectangles and triangles in $S^2$.}\label{esf}
\end{center}
\end{figure}

Let $m+1$ be the number of parallels and $n$ the number of meridians. We can assume $n\ge 3$. There exists a homeomorphism $h:|M_{n,m}|\to S^2$ that maps each of the $nm$ squares in the definition of $M_{n,m}$ homeomorphically to one of the spherical rectangles in $S^2$, and each triangle containing vertex $0$ or $1$ to a spherical triangle. Then for every $w\in M_{n,m}$, the diameter of $h (\ost (w))$ is smaller than $\delta$ and thus it is contained in a member of $\mathcal{U}$. Let $\varphi : M_{n,m}\to K$ be a simplicial approximation to $fh:|M_{n,m}|\to |K|$. In particular $|\varphi|$ is homotopic to $fh$. In the rest of the proof we will show that $|\varphi|$ is null-homotopic, and hence so is $f$.

The \textit{star cluster} of a simplex is the union of the closed stars of its vertices. If $m\ge 2$, the star cluster $D$ of $\{(1,m), (1,m-1)\}$ is a triangulation of $D^2$ (shaded region in Fig. \ref{mn} (a)).  Note that $D$ has $8$ vertices. We perform an $8$-starring to obtain a second triangulation $\widetilde{M}_{n,m}$ of $S^2$ replacing $D$ by $w\partial D$ (see Fig. \ref{mn} (b)), and a simplicial map $\widetilde{\varphi}:\widetilde{M}_{n,m} \to K$ which is null-homotopic if and only if $\varphi$ is null-homotopic by Lemma \ref{disco}. Now the star cluster of $\{(2,m),(2,m-1)\}$ is a triangulation of $D^2$ with $8$ vertices (Fig. \ref{mn} (c)) and we can make a new $8$-starring to obtain a new triangulation of $S^2$ (Fig. \ref{mn} (d)). We continue in this way until we replace the star cluster of $\{(n-1,m),(n-1,m-1)\}$ by a cone (to obtain a complex as in Fig. \ref{mn} (e)). Then the star cluster of $\{(0,m),(0,m-1)\}$ has $8$ vertices (Fig. \ref{mn} (e)) and we perform an $8$-starring to replace it by a new cone (Fig. \ref{mn} (f)). This complex is isomorphic to $M_{n,m-1}$.

If $m=1$, then a series of $8$-starrings allows us to replace $M_{n,1}$ by $M_{n,0}$, now using triangulations of $D^2$ of $8,7$ and $6$ vertices (Fig. \ref{mnp}). 

\begin{figure}[h] 
\begin{center}
(a) \includegraphics[scale=0.45]{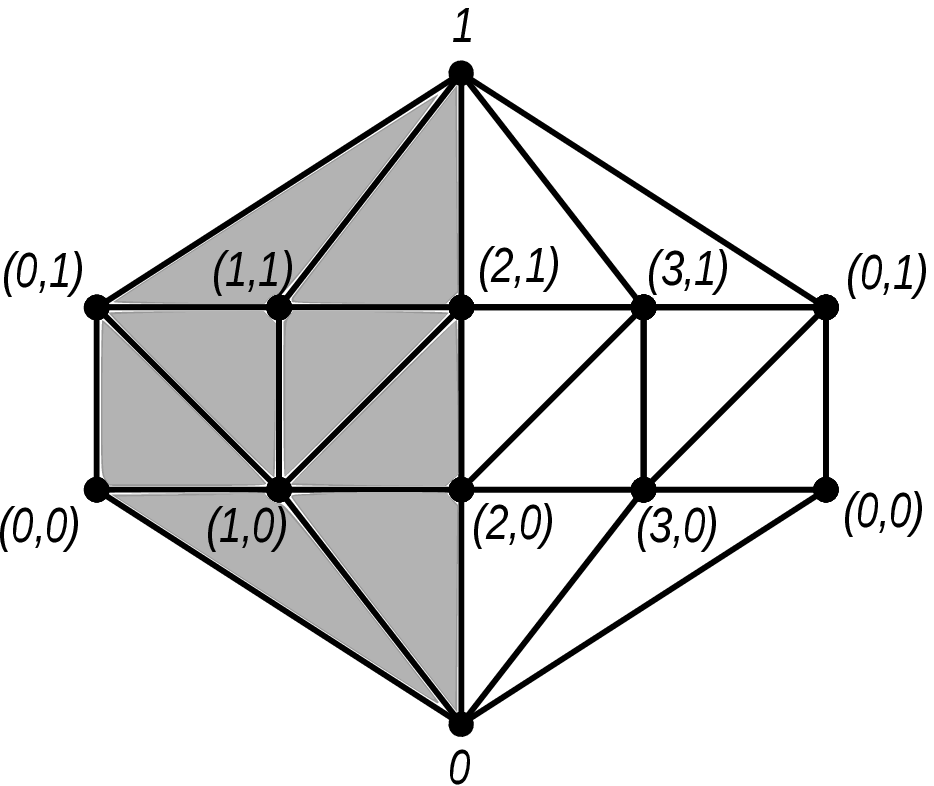}
(b) \includegraphics[scale=0.45]{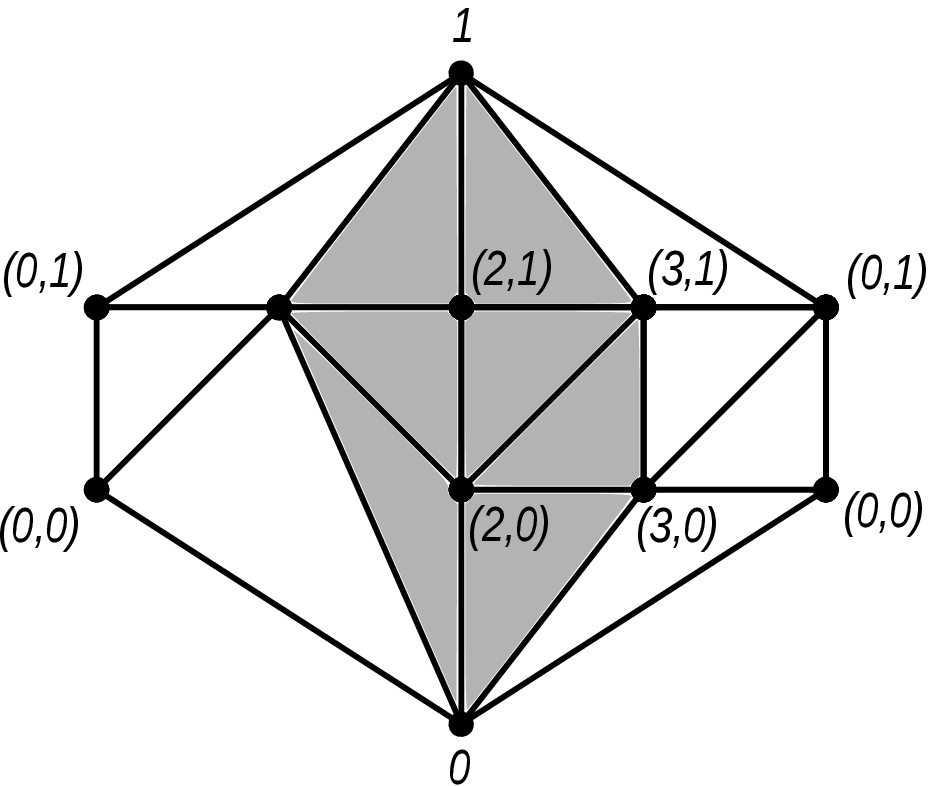}
(c) \includegraphics[scale=0.45]{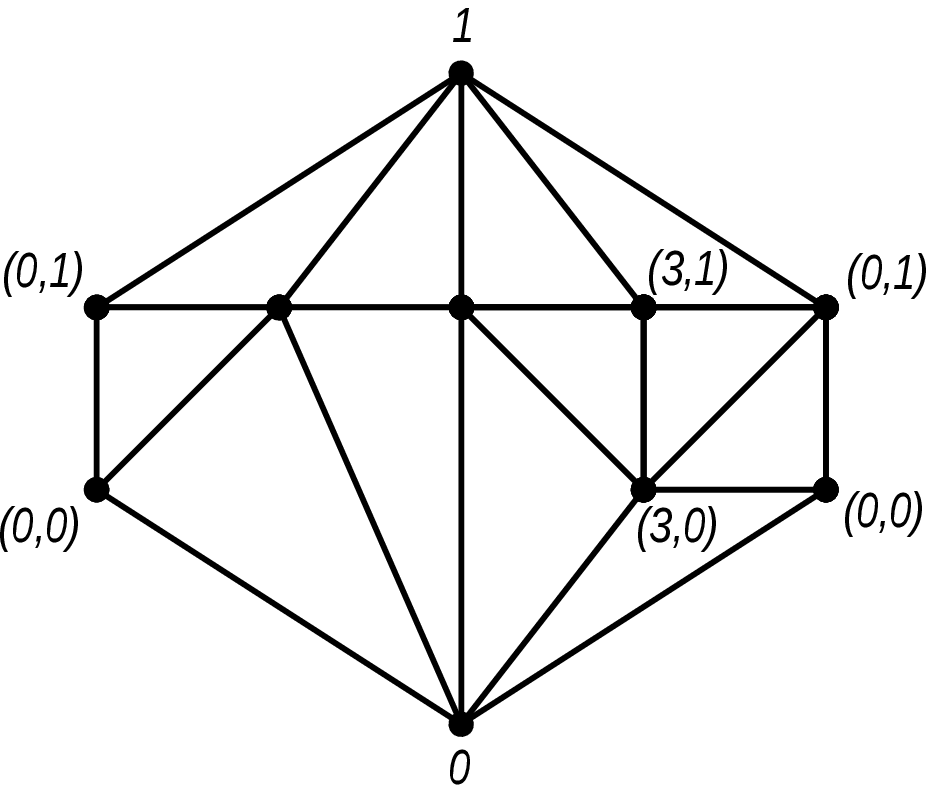}

(d) \includegraphics[scale=0.45]{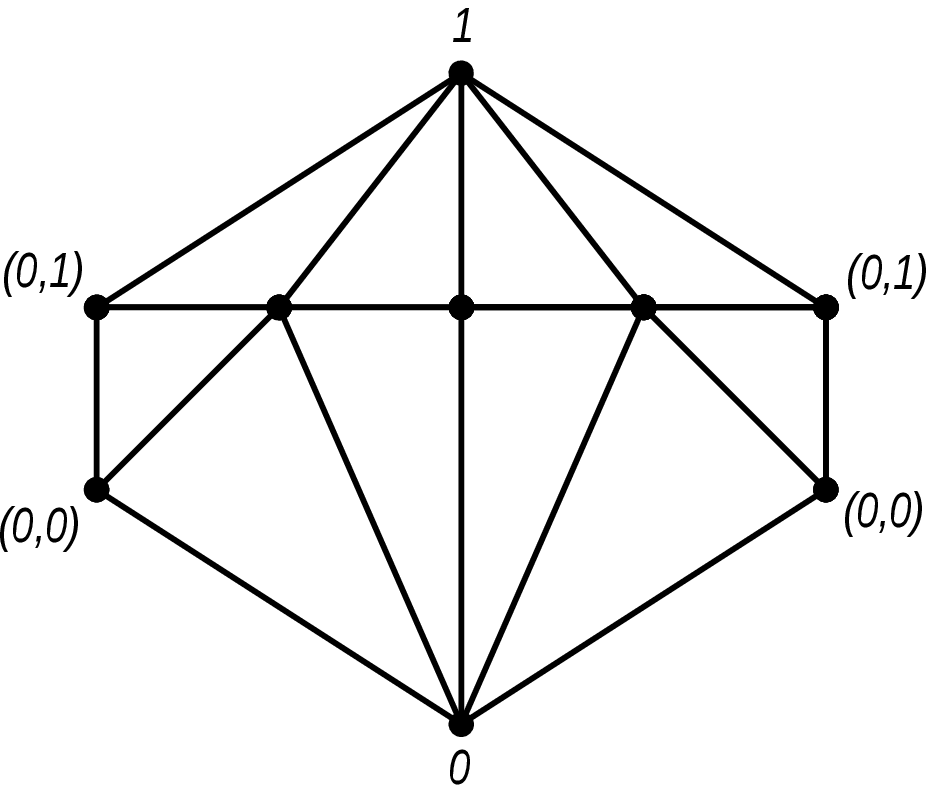}
(e) \includegraphics[scale=0.45]{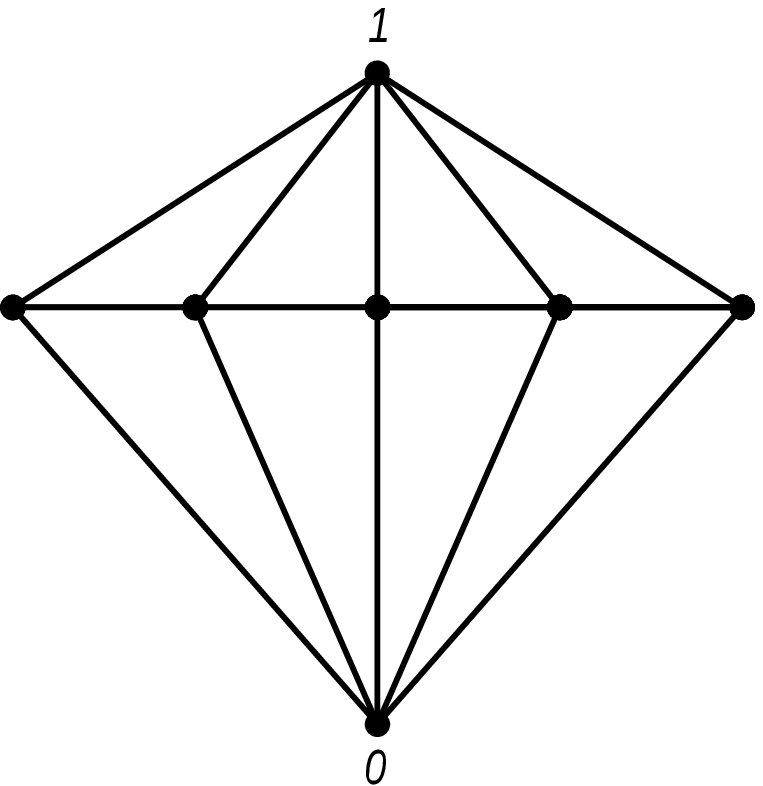}
\caption{From $M_{n,1}$ to $M_{n,0}$ with $8$-starrings.}\label{mnp}
\end{center}
\end{figure}

The complex $M_{n,0}$ is the suspension of a cycle of $n$ vertices. If $n\ge 4$, the star cluster of two adjacent vertices of the cycle is a triangulation of $D^2$ with $6$ vertices. When we replace this by the cone over the boundary we get a new triangulation of $S^2$, isomorphic to $M_{n-1,0}$ (Fig. \ref{mnpp}). Finally, note that $M_{3,0}$ has just $5$ vertices. Thus, the image of a simplicial map $M_{3,0}\to K$ is contained in a cone, and then it is null-homotopic. This proves that the original simplicial map $\varphi: M_{n,m}\to K$ is null-homotopic.

\begin{figure}[h] 
\begin{center}
(a) \includegraphics[scale=0.45]{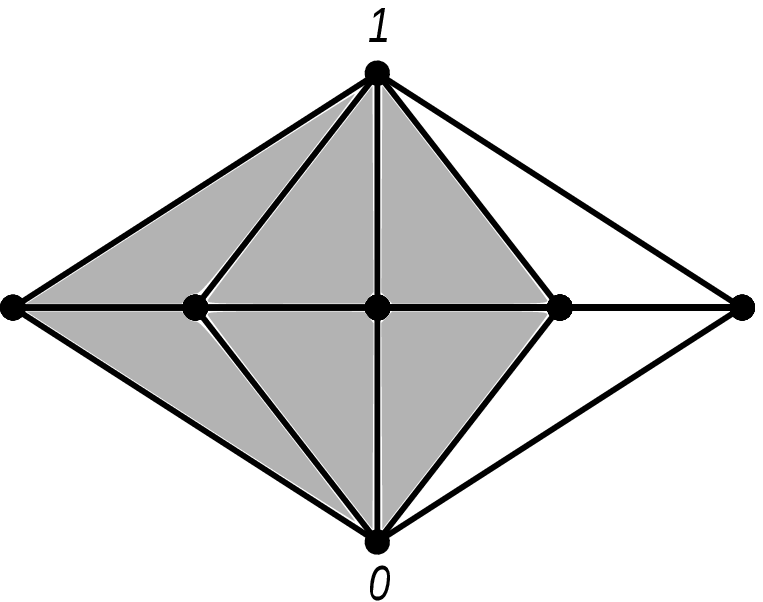}
(b) \includegraphics[scale=0.45]{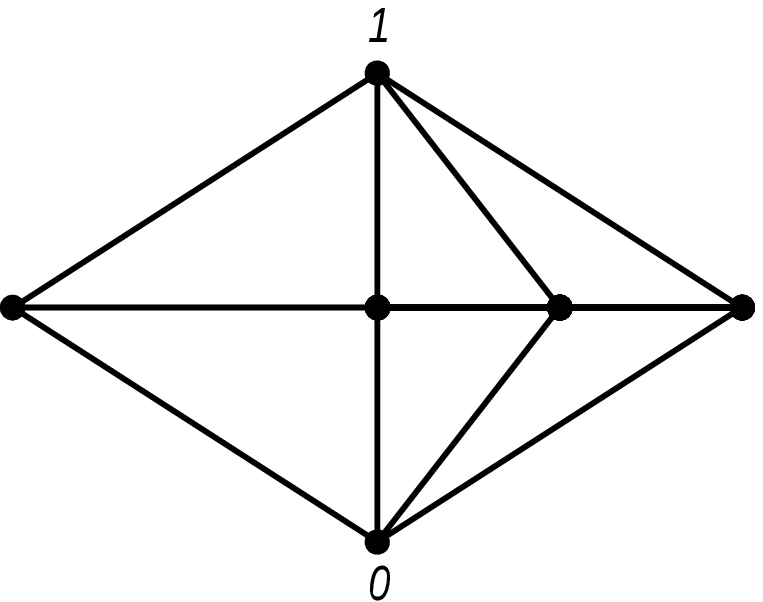}

\caption{$M_{n,0}$ $6$-stars to $M_{n-1,0}$.}\label{mnpp}
\end{center}
\end{figure}
\end{proof}

\section{Higher connectivity}

\begin{teo} \label{main}
Let $n\in \N$. If a simplicial complex is $6^{n}$-conic, then it is $n$-connected.
\end{teo}
\begin{proof}
Let $K$ be $6^n$-conic. Since $(6^{n})_{n\in \N}$ is increasing, we only have to prove that every map $S^n\to |K|$ is null-homotopic. As in the proof of Theorem \ref{teo2conexo} we will work with a family of triangulations of $S^n$. 

\medskip

\textbf{Triangulations of $S^n$, starrings and the basis of the induction.}

\medskip

Given $k\ge 0$ we consider a cubical complex $C_k$ homeomorphic to $D^n$. It is obtained by subdividing $I^n$ in $2^{nk}$ \textit{basic} cubes $Q_{i_1,i_2, \ldots, i_n}=[\frac{i_1-1}{2^k}, \frac{i_1}{2^k}]\times [\frac{i_2-1}{2^k}, \frac{i_2}{2^k}]\times \ldots \times [\frac{i_n-1}{2^k}, \frac{i_n}{2^k}]$ ($1\le i_j\le 2^k$ for each $1\le j\le n$) of dimension $n$ and side $\frac{1}{2^k}$. Let $C_k'$ be the barycentric subdivision of $C_k$. That is, a simplex of $C_k'$ is a set $\{b(Q_0), b(Q_1), \ldots , b(Q_l)\}$, where the $Q_i$ are cubes of $C_k$ (of different dimension), $Q_i$ is a face of $Q_{i+1}$ for each $0\le i<l$, and $b(Q)$ denotes the barycenter of the cube $Q$. Then $C_k'$ is a simplicial complex and we define $L_k$ to be a union of two copies of $C_k'$ identified by their boundary. Clearly $L_k$ is a triangulation of $S^n$. We will prove that for $k\ge 0$, $L_{k+1}$ $6^n$-stars to $L_k$. 

Let $Q^n=Q_{i_1,i_2, \ldots, i_n}$ be one basic cube of $C_{k}$. This is subdivided into $2^n$ basic cubes of $C_{k+1}$. Consider the subcomplex $C_{k+1}'[Q^n]$ of $C_{k+1}'$ of simplices contained in $Q^n$. This is a triangulation of $Q^n$, which is homeomorphic to $D^n$ (see Figure \ref{fk2}). The number of vertices of $C_{k+1}'[Q^n]$ is less than $2^n$ times the number of vertices $t_n$ in $C_{k+1}'$ of one basic cube of $C_{k+1}$. But the number $t_n$ is the number of faces of an $n$-dimensional cube, which is easily seen to be $3^n$: a face of $I_1\times I_2 \times \ldots \times I_n$ corresponds to choosing for each interval either the whole interval, its minimum or its maximum. Thus, the number of vertices of $C_{k+1}'[Q^n]$ is smaller than $6^{n}$, and we can perform a $6^{n}$-starring. The new vertex can be identified with the barycenter $b(Q^n)$ and its link is $\partial C_{k+1}'[Q^n]$. We do this for each basic cube of the two copies of $C_k$ in $L_k$ to obtain a new complex $L_{k+1}^{(1)}$.

\begin{figure}[h] 
\begin{center}
\includegraphics[scale=0.5]{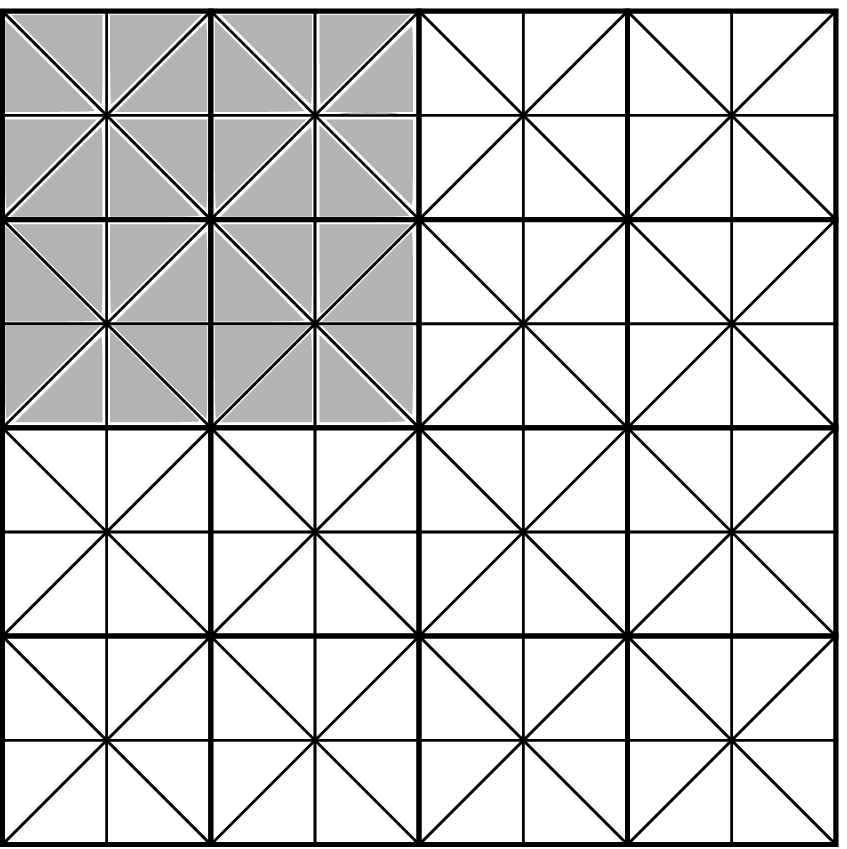}
\includegraphics[scale=0.5]{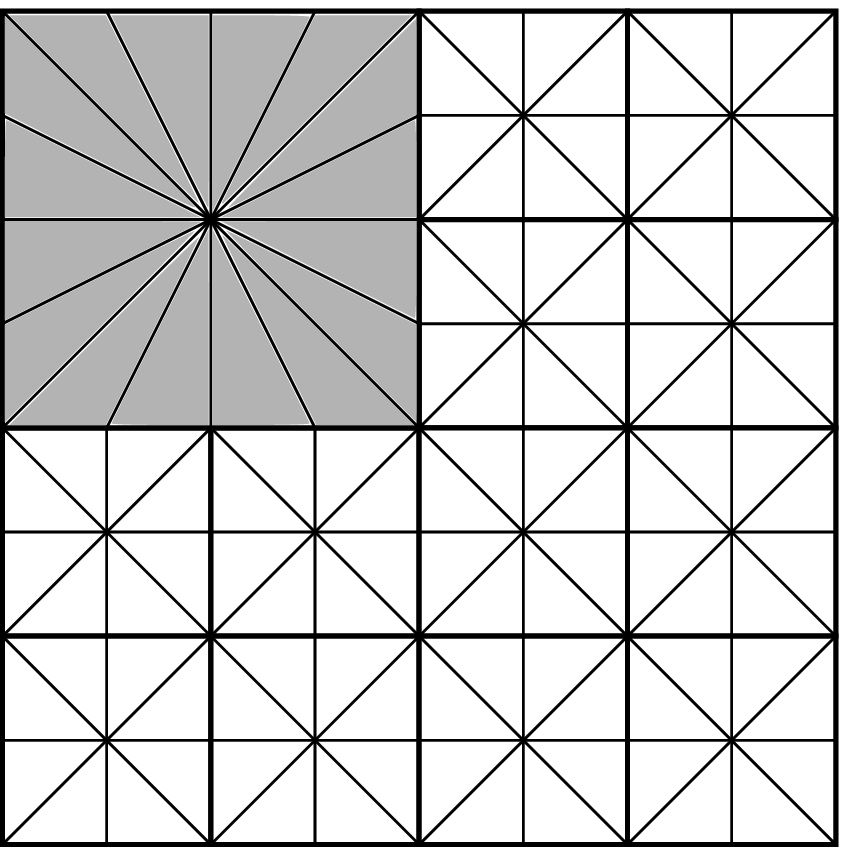}
\includegraphics[scale=0.5]{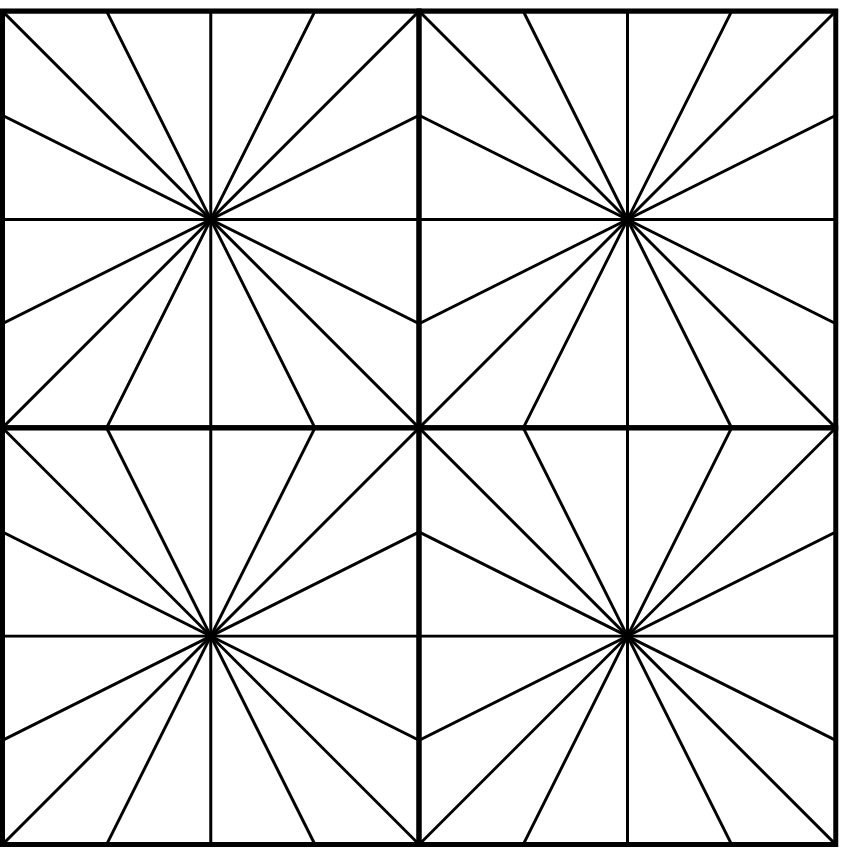}

\caption{The complex $L_2$ for $n=2$. Only one copy of $C_2'$ is visible in this picture. A basic cube $Q^2$ of $L_1$ is shaded. In the middle the first $6^2$-starring has been performed. At the right the complex $L_2^{(1)}$.}\label{fk2}
\end{center}
\end{figure}

If $Q^{n-1}$ is now an $(n-1)$-dimensional cube of $L_k$ (a codimension $1$ face of a basic cube of one of the two copies of $C_k$), then the subcomplex $L_{k+1}[Q^{n-1}]$ of $L_{k+1}$ of simplices in $Q^{n-1}$ is contained in the boundaries $\partial L_{k+1}[Q_1^n]$ and $\partial L_{k+1}[Q_2^n]$ of two basic cubes $Q_1^n$, $Q_2^n$ of $L_k$ (see Figure \ref{fk2q}). Thus $b(Q_1^n)L_{k+1}[Q^{n-1}]$ and $b(Q_2^n)L_{k+1}[Q^{n-1}]$ are two cones of $L_{k+1}^{(1)}$ whose union $\Sigma L_{k+1}[Q^{n-1}]$ is a triangulation of $D^n$. Moreover, the number of vertices of this suspension is less that $6^{n-1}+2\le 6^{n}$. Then $\Sigma L_{k+1}[Q^{n-1}]$ can be $6^{n}$-starred in $L_{k+1}^{(1)}$. The new vertex may be identified with $b(Q^{n-1})$. We do this for every $(n-1)$-cube of $L_k$ to obtain a new complex $L_{k+1}^{(2)}$. Note that the link of $b(Q^{n-1})$ is $b(Q_1^n)\partial L_{k+1}[Q^{n-1}] \cup b(Q_2^n)\partial L_{k+1}[Q^{n-1}] $. 

\begin{figure}[h] 
\begin{center}
\includegraphics[scale=0.5]{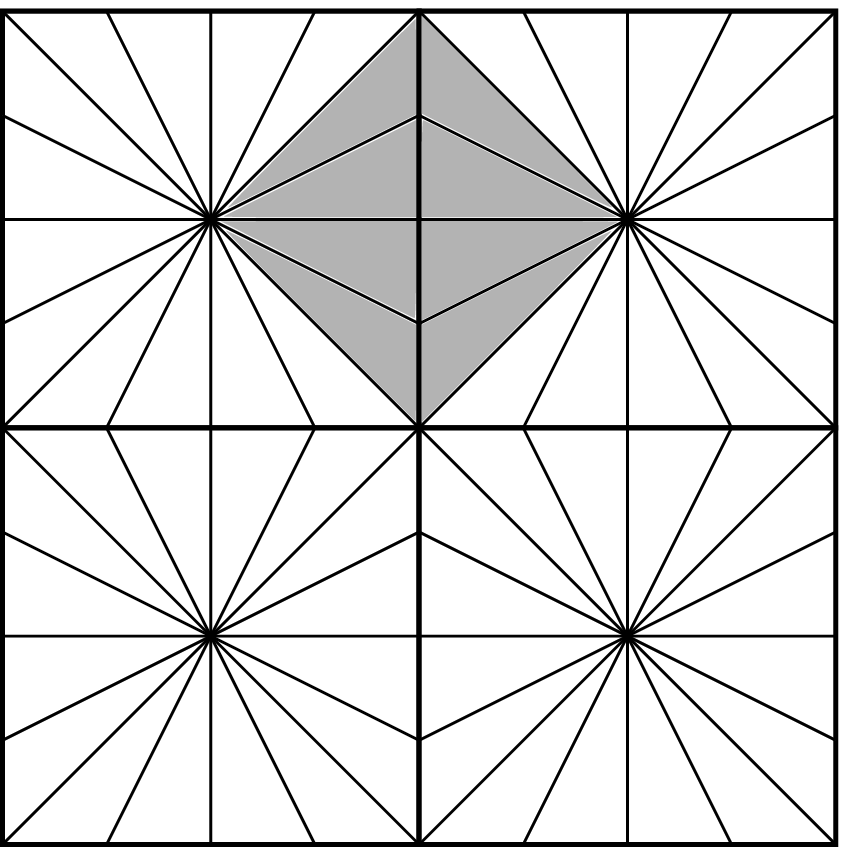}
\includegraphics[scale=0.5]{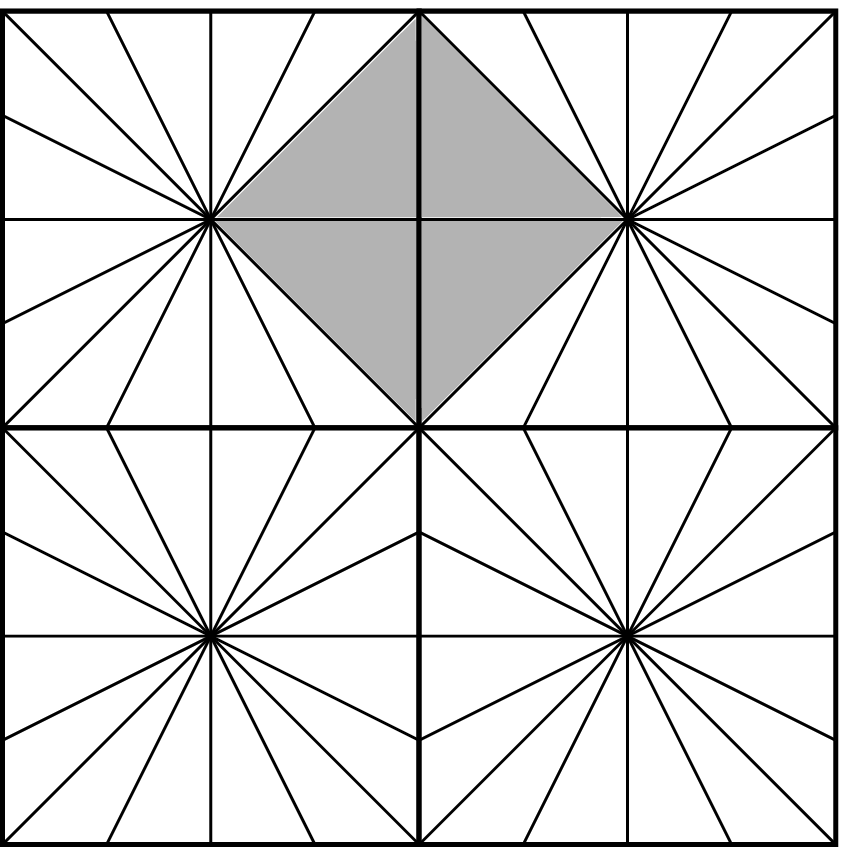}
\includegraphics[scale=0.5]{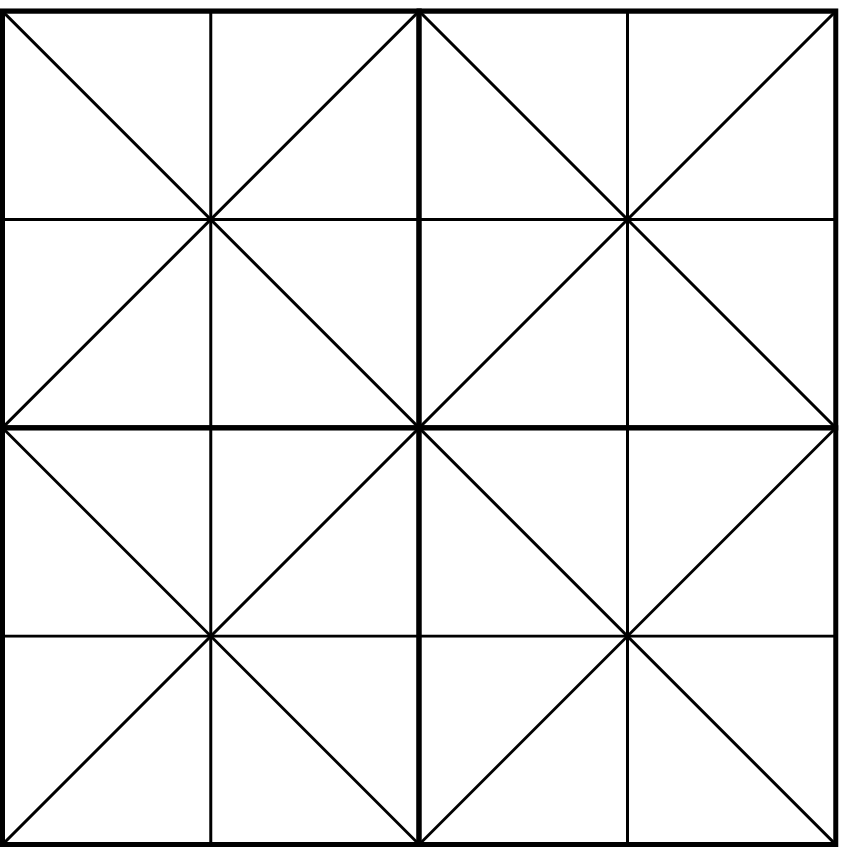}

\caption{The complex $L_2^{(1)}$ and the suspension of a cube $Q^1$ of $L_1$. In the middle, the result of a $6^2$-starring. At the right the complex $L_2^{(2)}$, isomorphic to $L_1$.}\label{fk2q}
\end{center}
\end{figure}

\medskip

\textbf{The induction step, combinatorial manifolds and shellability.}

\medskip

Suppose we have already $6^{n}$-starred all the $(n-p)$-cubes of $L_k$ in $L_{k+1}$ for $0\le p< q$ where $q\in \N$ is smaller than or equal to $n-1$, and we have obtained the complex $L_{k+1}^{(q)}$. Moreover, suppose that the maximal simplices of $L_{k+1}^{(q)}$ are of the form $\sigma b(Q^{n-q+1})b(Q^{n-q+2})\ldots b(Q^{n})$ where $\sigma$ is an $(n-q)$-simplex of $L_{k+1}$ inside an $(n-q)$-cube $Q^{n-q}$ of $L_k$, and $Q^i$ is an $i$-cube of $L_k$ for $n-q+1\le i\le n$, where $Q^{i}$ is a face of $Q^{i+1}$ for each $n-q\le i<n$. Let $Q^{n-q}$ be an $(n-q)$-cube of $L_k$. Then $L_{k+1}[Q^{n-q}]$ is a triangulation of $D^{n-q}$. The subcomplex $S$ of $L_{k+1}^{(q)}$ generated by the simplices $b(Q^{n-q+1})b(Q^{n-q+2})\ldots b(Q^{n})$ such that $Q^i\subsetneq Q^{i+1}$ for every $n-q\le i<n$, is contained in the star of each simplex of $L_{k+1}[Q^{n-q}]$. That is, the join $S*L_{k+1}[Q^{n-q}]$ is a subcomplex of $L_{k+1}^{(q)}$. On the other hand $|S*L_{k+1}[Q^{n-q}]|$ is homeomorphic to $|S*L_{k}[Q^{n-q}]|$, but $S*L_{k}[Q^{n-q}]$ is (isomorphic to) the star of the vertex $b(Q^{n-q})$ in $L_k$. Hence, in order to see that $S*L_{k+1}[Q^{n-q}]$ triangulates $D^n$, it suffices to show that $L_k$ is a combinatorial manifold of dimension $n$. The fact that $L_k$ is a PL manifold can be proved as follows. We prove more generally that if $m_1, m_2, \ldots, m_n\in \N$ and $C_{m_1,m_2,\ldots, m_n}$ is the cubic complex obtained by subdividing $I^n$ in $m_1\times m_2\times \ldots \times m_n$ $n$-cubes, then its barycentric subdivision is a combinatorial $n$-ball. We do this by induction in $n$ and all the $m_i$. For $n=1$ it is trivial. If we assume it is valid for $n-1$, then for $n$ is also valid: if all the $m_i$ are $1$, we have just one $n$-cube in $C=C_{m_1,m_2,\ldots, m_n}$. So we must show that the barycentric subdivision of the $n$-cube is a combinatorial ball. Since $\partial C$ is the boundary of a polytope, its barycentric subdivision $(\partial C)'$ is isomorphic to the boundary of a polytope (see for instance Ewald-Shephard's paper \cite[Section 2]{ES}). Now we invoke Bruggesser-Mani's Theorem \cite[Corollary]{BM} that boundaries of polytopes are shellable. Finally, a shellable $(n-1)$-dimensional pseudomanifold is a combinatorial sphere (see Bing \cite{Bin} or Danaraj-Klee \cite{DK}). Thus, the cone $C'$ over $(\partial C)'$ is a combinatorial ball. 

If now we have some $m_i\ge 2$, then $C=C_{m_1,m_2, \ldots,  m_n}$ is a union of two cubical complexes $C^1$ and $C^2$ isomorphic to $C_{m_1,m_2,\ldots, m_{i-1}, m_i-1, m_{i+1}, \ldots,  m_n}$ and $C_{m_1,m_2,\ldots, m_{i-1}, 1, m_{i+1}, \ldots,  m_n}$ and their intersection is isomorphic to $C_{m_1,m_2,\ldots, m_{i-1}, m_{i+1}, \ldots,  m_n}$. By induction $(C^1)'$ and $(C^2)'$ are combinatorial $n$-balls and $(C^1)'\cap (C^2)'$ is a combinatorial $(n-1)$-ball. By \cite[Theorem 14:3]{Ale} $C'$ is a combinatorial $n$-ball.

Finally, if we take two copies of $C=C_{m_1,m_2, \ldots,  m_n}$ and identify their boundaries, then we obtain a cubical complex whose barycentric subdivision is a combinatorial $n$-sphere (\cite[Theorem 14:1]{Ale}). Our complex $L_k$ is a particular case of these $C$ in which all the $m_i$ are equal to $2^k$.

Now that we have proved $L_k$ is a combinatorial $n$-manifold, we count the number of vertices in $S*L_{k+1}[Q^{n-q}]$. The complex $L_{k+1}[Q^{n-q}]$ has less than $6^{n-q}$ vertices. The vertices of $S$ are in correspondence with cubes of the two copies of $C_k$ containing $Q^{n-q}$ properly. To obtain an upper bound we can assume $Q^{n-q}$ is just a vertex of one copy $C_k$ (or both). If the vertex is not in the boundary of $C_k$, then it is contained in exactly $3^n-1$ cubes. If it is in the boundary, then it is certainly contained in less than $2(3^n-1)$ cubes. Thus the number of vertices in $S*L_{k+1}[Q^{n-q}]$ is smaller than $6^{n-q}+2.3^n$, which is less than $6^{n}$ since $n>q\ge 1$. Therefore, we can $6^{n}$-star $S*L_{k+1}[Q^{n-q}]$ and introduce a new vertex $b(Q^{n-q})$. We perform one such starring for every $(n-q)$-cube of $L_k$ to transform $L_{k+1}^{(q)}$ into $L_{k+1}^{(q+1)}$. Note that if $Q^{n-q}$ is an $(n-q)$-cube of $L_k$, the link of $b(Q^{n-q})$ in $L_{k+1}^{(q+1)}$ is $S*\partial L_{k+1}[Q^{n-q}]$, so the maximal simplices of $L_{k+1}^{(q+1)}$ containing $b(Q^{n-q})$ are of the form $\sigma b(Q^{n-q})b(Q^{n-q+1})\ldots b(Q^{n})$ where $\sigma$ is an $(n-q-1)$-simplex of $L_{k+1}$ inside an $(n-q-1)$-cube $Q^{n-q-1}$ of $L_k$, and $Q^i$ is an $i$-cube of $L_k$ for $n-q+1\le i\le n$, where $Q^{i}$ is a face of $Q^{i+1}$ for each $n-q-1\le i<n$. This concludes the induction step.

We have proved that $L_{k+1}$ can be $6^{n}$-starred to $L_{k+1}^{(n)}$, which is isomorphic to $L_k$.

\medskip

\textbf{A linear metric and small subdivisions.}

\medskip

Let $f: S^n \to |K|$ be a continuous map. It is not hard to show that there exists $k\in \N$ and a simplicial approximation $\varphi : L_k \to K$ of $f$. We give a concrete proof using the classic notion of linear metric (see \cite{Spa}). We consider two copies of $I^n$, $I_1$ and $I_2$, as metric spaces with the usual metric, which will be called $d_1$ and $d_2$. Let $J$ be the set which is the union of $I_1$ and $I_2$ in which the two copies of $\partial I^n$ have been identified. It is easy to define a metric $d$ in $J$ which restricts to $d_i$ in $I_i$ for $i=1,2$. One can define for $x\in I_1$, $y\in I_2$, $d(x,y)=min \{d_1(x,a)+d_2(a,y) | \ a \in \partial I^n\}$. This is well defined and it is a metric (easy exercise, see also \cite[Lemma 5.24]{Bri}). Also note that $J$ is homeomorphic to $S^n$: the topological space $I_1\bigcup\limits_{\partial I^n} I_2$ with the quotient topology $\tau$ is homeomorphic to $S^n$ and the identity $I_1\bigcup\limits_{\partial I^n} I_2 \to J$ is continuous by definition of $\tau$, and bijective. Since the domain is compact and the codomain Hausdorff, it is a homeomorphism. Let $h:J\to S^n$ be a homeomorphism. Note that $J=I_1\bigcup\limits_{\partial I^n} I_2$ is the cubic complex $C_0\bigcup\limits_{\partial C_0} C_0$ defined at the beginning of the proof. Given $k\in \N$, $L_{k}$ is a subdivision of $J$ and we consider $|L_k|$ as a metric space with the same metric $d$. We know that $fh:|L_k|\to |K|$ admits a simplicial approximation $L_k\to K$ if and only if for every $w\in L_k$ there exists $v\in K$ such that $\ost(w)\subseteq (fh)^{-1}(\ost(v))$. Now, the diameter of the cubes $I_1$, $I_2$ is $\sqrt{n}$ and the diameter of each basic cube in $L_k$ is then $\frac{\sqrt{n}}{2^k}$. Hence the diameter of each simplex in $L_k$ is bounded above by the same number and the diameter of $\st (w)$ is smaller than or equal to $\frac{\sqrt{n}}{2^{k-1}}$ for every $w\in L_k$. If $\delta >0$ is a Lebesgue number of $\mathcal{U}=\{(fh)^{-1}(\ost (v))\}_{v\in K}$, taking $k\in \N$ such that $\frac{\sqrt{n}}{2^{k-1}}<\delta$ guarantees there is a simplicial approximation $L_k\to K$ to $fh$. 

Since $L_k$ can be $6^{n}$-starred to $L_0$, there is a simplicial map $\psi: L_0 \to K$ which is null-homotopic if and only if $fh$ is null-homotopic. Now, $C_0'$ has $3^n$ vertices and $L_0$ has $3^n+1\le 6^{n}$ vertices, so the image of $\psi$ is contained in a cone and $\psi$ is null-homotopic. Thus, $fh$ and then $f$ are null-homotopic.      
\end{proof}

\begin{coro}
If a simplicial complex is $r$-conic for every $r\ge 0$, then it is contractible.
\end{coro}

\begin{obs}
In the proof of Theorem \ref{main} (also Theorem \ref{teo2conexo}), the simplicial approximation $L_k\to K$ to $f:S^n\to |K|$ extends to a simplicial map from a triangulation of $D^{n+1}$ where the number of internal vertices is the number of starrings performed in the proof plus one.

Indeed if $M$ is a simplicial complex and $D\leqslant M$ is a combinatorial $d$-ball which can be $r$-starred to obtain a complex $\widetilde{M}$ by removing the simplices in $D\smallsetminus \partial D$ and attaching the cone $w\partial D$, then $M$ and $\widetilde{M}$ are subcomplexes of $N=M\cup wD$. Since $\partial D$ is a combinatorial $(d-1)$-sphere, $N$ is obtained from $\widetilde{M}$ by attaching a combinatorial $(d+1)$-ball $wD$ and the intersection $\widetilde{M} \cap wD=w\partial D$ is a combinatorial $d$-ball. That is, $\widetilde{M}$ PL-expands to $N$, denoted $\widetilde{M} \nearrow N$, with an expansion of dimension $d+1$.


In the proof of Theorem \ref{main}, the simplicial map $\psi :L_0\to K$ extends to the cone $vL_0$ by conicity. The complex $vL_0$ is a combinatorial $(n+1)$-ball. Since $L_k$ $6^n$-stars to $L_0$, by the previous paragraph there is a sequence of PL-expansions $L_0=A_0\nearrow A_1 \nearrow A_2 \nearrow \ldots \nearrow A_m \supseteq L_k$, each of dimension $d+1=n+1$. By coning over $L_0$ in each complex $A_i$ we obtain PL-expansions $vL_0\nearrow B_1 \nearrow B_2 \nearrow \ldots \nearrow B_m$ and by \cite[Theorem 14:3]{Ale} every complex $B_i$ is a combinatorial $(n+1)$-ball. It is easy to see that $L_k=\partial B_m$. Moreover each complex in the sequence of starrings from $L_k$ to $L_0$ is a subcomplex of $B_m$ and there is a simplicial map $B_m\to K$ which extends all the simplicial maps from these complexes to $K$ which implicitly appear in the proof. The ball $vL_0$ has a unique internal vertex and each expansion adds one internal vertex. See Figure \ref{expansion}.

\begin{figure}[h] 
\begin{center}
\includegraphics[scale=0.26]{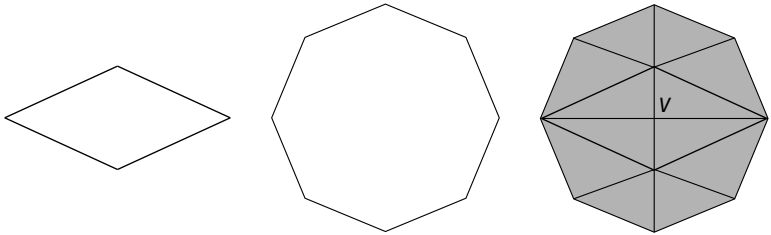}

\caption{For $n=k=1$, we have $L_0$, $L_1$ and the disk $B_2$ with $3$ internal vertices.}\label{expansion}
\end{center}
\end{figure}

\end{obs}

\section{Connectivity of random complexes in the medial regime}

Linial and Meshulam \cite{LM} propose the following model. Given $n\in \N$ and $0\le p\le 1$, consider the complete graph on $n$ vertices and add each $2$-simplex independently with probability $p$. Meshulam and Wallach \cite{MW} generalize this to arbitrary dimension $d$ by taking the complete $(d-1)$-skeleton and adding each $d$-simplex with probability $p$. In Kahle's clique-complex model \cite{Kah1} a random graph in $n$ vertices is taken by adding each edge independently with probability $p$ and then considering the clique complex of this graph. These two models have been useful to describe different situations, but note that they leave many simplicial complexes out of the picture. In particular the cases they study exclude one another: for $d\ge 2$ there is only one clique complex of dimension $d$ in $n$ vertices with complete $(d-1)$-skeleton. Also, the clique-complex model goes against the philosophy of \cite{BCI} that networks should be modeled taking into account interactions which are not determined by pairwise interactions.

The medial regime model can be described as follows. Pick some $0<p<P<1$ and for every non-empty subset $\sigma$ of a set $V$ of $n$ points define a probability $p\le p_{\sigma} \le P$. Now construct the $0$-skeleton of a simplicial complex by adding each vertex $v\in V$ with probability $p_v$. Then construct the $1$-skeleton by adding a $1$-simplex $\sigma$ whose boundary is contained in the $0$-skeleton with probability $p_{\sigma}$, and so on. In this way every simplicial complex with vertex set contained in $V$ has a positive probability to occur. As recalled in the introduction in the medial regime the probability of a complex being simply connected tends to $1$ as $n\to \infty$ (\cite{FM}) and the Betti numbers of a random simplicial complex vanish in dimension smaller than or equal to a fixed dimension $d$ with probability $1$ as $n\to \infty$. In \cite{EZFM} it is proved that a random complex is $2$-connected with probability $1$ as $n\to \infty$. The key result they prove is the following

\begin{teo} [Even-Zohar, Farber, Mead] \cite[Proposition 5.1]{EZFM}
For every integer $r\ge 1$, the probability that a medial regime random simplicial complex is $r$-ample tends to one, as $n\to \infty$.
\end{teo}

From this and Theorem \ref{main} we deduce the following

\begin{coro} \label{proba}
Let $d\ge 0$. In the medial regime the probability of a complex being $d$-connected tends to $1$ as $n\to \infty$. 
\end{coro}


\begin{thebibliography}{99}

\bibitem{Ale} J.W. Alexander. \textit{The combinatorial theory of complexes}. Ann. of Math. (2) 31(1930), no. 2, pp. 292-320.

\bibitem{Sci} A. Barabási. \textit{Emergence of Scaling in Random Networks}. Science, 286(1999), pp. 509-512.

\bibitem{BCI} F. Battiston, G. Cencetti, I. Iacopini, V. Latora, M. Lucas, A. Patania, J.G. Young, G. Petri. \textit{Networks beyond pairwise interactions; structure and dynamics}. Physics Reports, 2020.

\bibitem{Bin} R.H. Bing. \textit{Some aspects of the topology of $3$-manifolds related to the Poincaré conjecture}. In: Lectures on Modern Mathematics II (T.L. Saaty, ed.), Wiley, New York 1964, pp. 93-128.

\bibitem{Bri} M.R. Bridson, A. Haefliger. \textit{Metric spaces of non-positive curvature}. Grundlehren der Mathematischen Wissenschaften 319, Springer-Verlag, Berlin, 1999.

\bibitem{BM} H. Bruggesser, P. Mani. \textit{Shellable decompositions of cells and spheres}. Math. Scand. 29(1971), pp. 197-205.

\bibitem{CF} A. Costa, M. Farber. \textit{Random Simplicial Complexes}. In: Configuration Spaces, Springer INdAM Ser., vol. 14, Springer, 2016, pp.129-153.

\bibitem{DK} G. Danaraj, V. Klee. \textit{Shellings of spheres and polytopes}. Duke Math. 1. 41(1974), pp. 443-451.

\bibitem{EZFM} C. Even-Zohar, M. Farber, L. Mead. \textit{Ample simplicial complexes}. arXiv:2012.01483, 2020.

\bibitem{ES} G. Ewald and G. C. Shephard. \textit{Stellar subdivisions of boundary complexes of convex polytopes}., Mathematische Annalen 210(1974), pp. 7-16.

\bibitem{FM} M. Farber, L. Mead. \textit{Random simplicial complexes in the medial regime}. Topology and its Applications, 107065, 2020.

\bibitem{FMS} M. Farber, L. Mead, L. Strauss. \textit{The rado simplicial complex}. arXiv:1912.02515, 2019.

\bibitem{Gla} L. Glaser. \textit{Geometrical combinatorial topology, Vol. 1}. Van Nostrand Reinhold Mathematical Studies 27, 1970.

\bibitem{Iac} I. Iacopini, G. Petri, A. Barrat, V. Latora. \textit{Simplicial models of social contagion}. Nature communications 10(1)(2019), pp. 1-9.

\bibitem{Kah1} M. Kahle. \textit{Topology of random clique complexes}. Discrete Math., 309(2009), pp. 1658-1671.

\bibitem{Kahleviejo} M. Kahle. \textit{Topology of random simplicial complexes: a survey}. In Algebraic topology: applications and new directions, vol. 620 of Contemp. Math., pp. 201-221, AMS, Providence, 2014.

\bibitem{Kahlenuevo} M. Kahle. \textit{Random simplicial complexes}. In Handbook of Discrete and Computational Geometry, CRC, Boca Raton, FL, 2017, pp. 581-603.

\bibitem{LM} N. Linial, R. Meshulam. \textit{Homological connectivity of random $2$-complexes}. Combinatorica, 26(2006), pp. 475-487.

\bibitem{Mes} R. Meshulam. \textit{Domination numbers and homology}. Journal of Combinatorial Theory, Series A 102(2003), pp. 321-330.

\bibitem{MW} R. Meshulam, N. Wallach. \textit{Homological connectivity of random $k$-dimensional
complexes}. Random Structures Algorithms, 34(2009), pp. 408-417.

\bibitem{Mun}  J.R. Munkres. \textit{Elements of algebraic topology}. Addison-Wesley Publishing Company, Menlo Park, CA, 1984.

\bibitem{Spa} E. Spanier. \textit{Algebraic Topology}. McGraw-Hill, New York, 1966.

\bibitem{Nat} D.J. Watts, S.H. Strogatz. \textit{Collective dynamics of “small-world” networks}. Nature, 393(1998), pp. 440-442.



\end{thebibliography}
\end{document}